\documentclass[11pt]{article}
\usepackage{amsmath}
\usepackage{amsthm}
\usepackage{amsfonts}
\usepackage{amssymb}
\usepackage{xcolor}
\usepackage{mathtools}
\usepackage[latin1]{inputenc}

\newtheorem{theorem}{Theorem}[section]
\newtheorem{lemma}[theorem]{Lemma}
\newtheorem{coro}[theorem]{Corollary}

\newtheorem{definition}[theorem]{Definition}
\newtheorem{remark}[theorem]{Remark}
\newtheorem{example}[theorem]{Example}

\newcommand{\R}{\mathbb{R}}

\newcommand{\N}{\mathbb{N}}

\begin{document}
\begin{center}
\textbf{\Large On $L^p$-spaces of functions with values in locally convex spaces}\\[2mm]
\textbf{Matthieu F. Pinaud\footnote{Supported by DICYT-USACH (grant 042632PC-POSTDOC)}, Humberto Prado}
\end{center}

\begin{abstract}
\noindent
We study Lusin-measurable functions with values in locally convex spaces. In particular, the behavior of pointwise limits of sequences of Lusin-measurable functions and exhibit pathological phenomena arising in the nonmetrizable setting. Moreover, we establish approximation and density results for $L^p$-spaces constructed with this notion of measurability, including the density of simple functions in Hausdorff locally convex spaces and convergence results obtained through dyadic approximations.
\end{abstract}

\section{Introduction}
    
Let $a<b$ be real numbers. For the closed interval $[a,b]$, we consider the Borel $\sigma$-algebra $\mathcal{B}([a,b])$ and the Lebesgue $\sigma$-algebra $\mathcal{L}([a,b])$. We denote the Lebesgue measure on $[0,1]$ by $\lambda$. For a real Banach space $E$, the classical theory of Lebesgue spaces $L^p([a,b],E)$ of vector-valued functions is well understood via Bochner-measurability.\\ 

\noindent
Many problems in infinite-dimensional analysis naturally arise in the setting of locally convex spaces. For example, the spaces encountered in distribution theory are, in general, non-metrizable locally convex spaces. Furthermore, solution curves of quantum stochastic differential equations take values in spaces of continuous linear operators between nuclear locally convex spaces (see, e.g., \cite{HRg}). Likewise, the theory of regularity of Lie groups (and half-Lie groups) is concerned with $L^p$
functions taking values in sequentially complete locally convex Lie algebras.
(see e.g. \cite{GL2, GL1, GaN, Nin1, Pin}).\\

\noindent
Let $E$ be a real topological vector space, $E'$ its topological dual space and $\mathcal{B}(E)$ its Borel $\sigma$-algebra. We recall  different notions of measurability for a function $\gamma:[a,b]\to E$.

\begin{itemize}
        \setlength\itemsep{0cm}
            \item[a)] The function $\gamma$ is Borel-measurable if $\gamma^{-1}(A) \in \mathcal{B}([a,b])$ for each $A\in \mathcal{B}(E)$.
            \item[b)] The function $\gamma$ is Lebesgue-measurable if $\gamma^{-1}(A) \in \mathcal{L}([a,b])$ for each $A\in\mathcal{B}(E)$.
            \item[c)] The function $\gamma$ is Lusin-measurable if for each $\varepsilon>0$ there exists a compact subset $K_\varepsilon$ of $[a,b]$ such that $\gamma|_{K_\varepsilon}$ is continuous and $\lambda\left ([a,b]\setminus K_\varepsilon\right)\leq \varepsilon$.
            \item[d)] The function $\gamma$ is strongly-measurable if it is the pointwise limit a.e. of a sequence of Lebesgue-measurable simple functions $(\beta_n)_n$. 
            \item[e)] The function $\gamma$ is weakly-measurable if for each $f\in E'$, the function $f\circ \gamma:[a,b]\to \mathbb{R}$ is Lebesgue-measurable. 
\end{itemize}

\noindent
In the classical case, when $E$ is finite-dimensional, Lebesgue, Lusin, strongly and weakly-measurability are all equivalents. If $E$ is any second countable topological vector space, by the Lusin's Theorem, Borel-measurability implies Lusin-measurability (see e.g. \cite{Anal,TTa}). If $E$ is a Banach space, by the Pettis Measurability Theorem (see \cite{Anal}), a function is strongly-measurable if and only if it is separably valued and weakly-measurable. We notice that, whenever $E$ is a notmetrizable locally convex space, then these notions are not equivalent since locally convex spaces introduce several pathologies that do not arise in the Banach setting. For instance, the sum of two Borel-measurable functions may fail to be a Borel-measurable function (see e.g. \cite{Stone}) and the pointwise limit of continuous functions may fail to be Lusin-measurable (see Remark \ref{pointwise}).\\

\noindent
The locally convex setting has already been studied by several authors (see e.g. \cite{FPM, PBr, GL2, GL1, Lewis, Nin1, RPa}). More recently, Natalie Nikitin has proved that, for every Lusin-measurable function $\gamma$ there exists a Borel-measurable function $\beta$ such that $\gamma=\beta$ a.e.\\

\noindent
In the case of Banach spaces, the classical theory of $L^p$-spaces is well understood. For example, for Fr\'echet spaces, in \cite{GL2}, H. Gl\"ockner has established classical results for $L^p$-spaces using Borel-measurability. In the locally convex setting, several authors have studied the construction of $L^p$-spaces using Lusin-measurability (see e.g. \cite{FPM, PBr, Nin1, Pin, RPa}).\\

\noindent
The article has been divided in two parts. The first one focuses on the study of the pointwise limit of a sequence of Lusin-measurable functions $(\gamma_n)_n$ and the conditions under which the pointwise limit is Lusin-measurable. \\

\noindent
In the second part, for $1\leq p< \infty$, we focus on the $L^p$-space of Lusin-measurable functions. We show some classical density and convergence results. We exhibit a pathology where the convergence in $L^p$ does not imply any kind of pointwise convergence, in opposition to the classical theory. Moreover, we show that every $L^p$-function is the limit of a sequence of simple functions defined by dyadic partition.\\

\noindent
Hereafter, we assume that all topological vector spaces and locally convex spaces are real vector spaces and are not necessarily Hausdorff. For details on locally convex spaces, we recall references \cite{GaN, Rud}. \\

\noindent
We recall the real-valued function $g:[0,1]\to [a,b]$, $g(t)=a+t(b-a)$. Now, if $\gamma:[a,b]\to E$ is Lusin-measurable, then the function $\gamma\circ g:[0,1]\to E$ is Lusin-measurable (see, \cite{Nin1}). This enables us to work with the closed interval $[0,1]$.

\section{Lusin-measurable functions}

 Let $E$ be a topological vector space and $\gamma:[0,1]\to E$ be a function. Then it is known that if $\gamma$ is strongly-measurable, then is weakly-measurable (see e.g. \cite{Lewis}). We notice that if $\gamma$ is Lusin-measurable, then is weakly-measurable. Indeed, Let $\varepsilon>0$ and $f\in E'$, then there exists a compact subset $K_\varepsilon$ of $[0,1]$ such that $\lambda([0,1]\setminus K_\varepsilon)<\varepsilon$ and $(f\circ\gamma)|_{K_\varepsilon} = f\circ (\gamma|_{K_\varepsilon})$ is continuous. Then $f\circ \gamma:[0,1]\to \mathbb{R}$ is Lusin-measurable and hence is Lebesgue-measurable.\\

\noindent
We emphasize the necessity to work with Lebesgue-measurable functions instead of Borel-measurable functions, because even for real valued-functions, there exists Lusin-measurable functions that are not Borel-measurables. Indeed, let $C\subseteq [0,1]$ be the Cantor set. By cardinality arguments, there exists a Lebesgue measurable subset $A\subseteq C$ which is not Borel measurable (see e.g. \cite[Exercise 1.14]{cantor}). We define the function  
\[ \gamma:[0,1]\to \R,\quad t\mapsto \chi_A(t).\]
    Then $\gamma$ is not Borel-measurable. To show $\gamma$ is Lusin-measurable, we consider $\varepsilon>0$. By regularity of the Lebesgue measure, we consider an open subset $U_\varepsilon$ of $[0,1]$ such that $C\subseteq U_\varepsilon$ and $\lambda(U_\varepsilon)<\varepsilon$. Then $K_\varepsilon=[0,1]\setminus U_\varepsilon$ is a compact subset of $[0,1]$ with $\lambda([0,1]\setminus K_\varepsilon) < \varepsilon$ and 
    \[ \gamma|_{K_\varepsilon}(t) = \chi_{A\cap K_\varepsilon}(t) = \chi_\emptyset(t) = 0,\quad \forall t\in K_\varepsilon.\]
    Therefore, $\gamma|_{K_\varepsilon}$ is continuous.\\ 
    
\noindent
Let $E$ and $F$ be topological vector spaces. It is straightforward that each continuous functions is Lusin-measurable. If $\gamma,\eta:[0,1]\to E$ are Lusin-measurable functions and $\varphi:E\to F$ is a continuous map, then $\varphi\circ \gamma:[0,1]\to F$ is Lusin-measurable. Since $E$ is a topological vector space, then for each $\alpha\in \mathbb{R}$, the function
\[ \alpha\gamma+\eta:[0,1]\to E,\quad t\mapsto \alpha\gamma(t)+\eta(t)\]
is Lusin-measurable. Therefore, the set of Lusin-measurable functions form a vector space.

\begin{lemma}\label{simple}
    Let $E$ be a topological vector space, $A\in\mathcal{L}([0,1])$ and $y_0\in E$. If $\chi_A$ denotes the characteristic function of $A$, then the function \[\beta:[0,1]\to E,\quad t\mapsto y_0\chi_A(t)\] 
    is Lusin-measurable.
\end{lemma}
\begin{proof}
    For $\varepsilon>0$, we must obtain a compact set $K_\varepsilon\subseteq [0,1]$ such that $\lambda\left( [0,1]\setminus K_\varepsilon\right)<\varepsilon$ and 
    \[ \beta|_{K_\varepsilon}:K_\varepsilon\to E,\quad t\mapsto \beta|_{K_\varepsilon}(t)= \left\{\begin{array}{ll}
       y_0,  & \text{ if $t\in A\cap K_\varepsilon$}  \\
       0,  & \text{ if $t\in K_\varepsilon\setminus A$},
    \end{array}\right.\] 
    is continuous, i.e., the set $A\cap K_\varepsilon$ is open and closed in $K_\varepsilon$.\\
    By regularity of the Lebesgue measure, for $\varepsilon>0$, there exists a compact subset $F_1$ of $[0,1]$ with $F_1\subseteq A$ such that
    \[ \lambda(A\setminus F_1) < \varepsilon/2.\] 
    and a compact subset $F_2$ of $[0,1]$ with $F_2\subseteq [0,1] \setminus A$ such that 
    \[ \lambda\left( \left( [0,1]\setminus A\right)\setminus F_2 \right) < \varepsilon/2. \]
    Next, we define $K_\varepsilon=F_1\cup F_2$. Thus
    \begin{align*}
        [0,1]\setminus K_\varepsilon 
        & \subseteq [A\setminus F_1]\cup [ ([0,1]\setminus A)\setminus F_2].
    \end{align*}
    Hence
    \[ \lambda([0,1]\setminus K_\varepsilon)\leq \lambda(A\setminus F_1)+\lambda(([0,1]\setminus A)\setminus F_2) < \varepsilon.\] 
    Since $F_1$ and $F_2$ are closed subsets of $K_\varepsilon$, the set $A\cap K_\varepsilon$ is closed in $K_\varepsilon$. We notice that
    \[ A\cap K_\varepsilon = F_1.\]
    Then $F_1$ is open in $K_\varepsilon$, since 
    \[ K_\varepsilon \setminus (A\cap K_\varepsilon) = K_\varepsilon \cap F_1^c = F_2.\] 
    Therefore, $\beta|_{K_\varepsilon}$ is continuous.    
\end{proof}

\begin{definition}
    Let $E$ be a topological vector space. A function $\beta:[0,1]\to E$ is called a simple function if there exists $y_1,...,y_n\in E$ and $A_1,...,A_n\in \mathcal{L}([0,1])$ such that
    \[ \beta=\sum_{i=1}^n y_i\chi_{A_i}.\]
    By the previous lemma, each simple function is Lusin-measurable. We denote the vector space of simple function by $\mathcal{S}([0,1],E)$.
\end{definition}

\begin{lemma}\label{Ksequence}
    Let $E$ be a topological vector space and $\gamma_n:[0,1]\to E$  a sequence of Lusin-measurable functions and let $\varepsilon>0.$ Then there exists a compact subset $H_\varepsilon$ of $[0,1]$ such that $\lambda([0,1]\setminus H_\varepsilon)<\varepsilon.$ Moreover, for each $n\in \mathbb{N}$, the function $\gamma_n|_{H_\varepsilon}$ is continuous.
\end{lemma} 
\begin{proof}    
Given $n\in\mathbb{N}$, we consider a compact subset $K_{\varepsilon,n}\subseteq [0,1]$ such that $\lambda([0,1]\setminus K_{\varepsilon,n})<{\varepsilon}/{2^n}$ and $\gamma_n|_{K_{\varepsilon,n}}$ is continuous. Next we define the set
    \[ H_\varepsilon = \bigcap_{n=1}^\infty K_{\varepsilon,n}.\] 
    Thus $H_\varepsilon$ is compact in $[0,1]$ and 
    \begin{align*}
        \lambda\left( [0,1]\setminus H_\varepsilon \right) = \lambda\left(  \bigcup_{n=1}^\infty  [0,1]\setminus K_{\varepsilon,n} \right) \leq \sum_{n=1}^\infty \lambda([0,1]\setminus K_{\varepsilon,n})  \leq \sum_{n=1}^\infty \frac{\varepsilon}{2^n} = \varepsilon.
    \end{align*}
    Moreover, since $H_\varepsilon\subseteq K_{\varepsilon,n}$, each $\gamma_n|_{H_\varepsilon}$ is continuous.
\end{proof}
\begin{theorem}\label{equi}
    Let $E$ be a locally convex space. A function $\gamma:[0,1]\to E$ is Lusin-measurable if, and only if, for each $\varepsilon>0$, there exists $A_\varepsilon\in \mathcal{L}([0,1])$ such that $\lambda([0,1]\setminus A_\varepsilon)<\varepsilon$ and $\gamma|_{A_\varepsilon}$ is uniform limit of simple functions. 
\end{theorem}
\begin{proof}
    Let $\gamma$ be a Lusin-measurable function and $\varepsilon>0$. There exists a compact subset $K_\varepsilon$ of $[0,1]$ such that $\lambda([0,1]\setminus K_\varepsilon)<\varepsilon$ and $\gamma|_{K_\varepsilon}$ is continuous. For $n\in \N$, we consider a partition $\{t_0,...,t_n\}$ of $[0,1]$ such that 
    \[ \max_{i\in \{1,...,n\}} \lvert t_i - t_{i-1}\lvert \leq 1/n.\] 
    For $i\in \{1,...,n\}$, if $[t_{i-1},t_i)\cap K_\varepsilon\neq \emptyset$, let us consider any $s_i\in [t_{i-1},t_i)\cap K_\varepsilon$ and define the function $\beta_{n}:[0,1]\to E$ by
    \[ \beta_{n}(t) = \sum\limits_{\substack{1\leq i\leq n \\ [t_{i-1},t_i)\cap K_\varepsilon\neq \emptyset}} \gamma(s_i)\chi_{[t_{i-1},t_i)\cap K_\varepsilon}(t) + \gamma(1)\chi_{\{1\}}(t),\quad \forall t\in [0,1].\] 
    Therefore, this sequence $(\beta_n)_n$ of simple functions converges uniformly on $K_\varepsilon$ to $\gamma$. Indeed, let $\tau>0$, since $\gamma|_{K_\varepsilon}$ is uniformly continuous, for each continuous seminorm $q$ on $E$, there exists $\delta_{\tau,q}>0$ such that 
    \[ \lvert t-s\lvert <\delta_{\tau,q}\implies q(\gamma|_{K_\varepsilon}(t)-\gamma|_{K_\varepsilon}(s))<\tau.\] 
    Let us consider $n_{\tau,q}\in \N$ such that
    \[ \frac{1}{n}<\delta_{\tau,q},\quad \forall n\geq n_{\tau,q}.\] 
    Therefore, for each $t\in K_\varepsilon\setminus\{1\}$, there exists some $j\in \{1,...,n\}$ such that $t,s_j\in [t_{j-1},t_j]$. Hence, if $\lvert t-s_j\lvert <1/n<\delta_{\tau,q}$, we have 
    \[ q(\gamma|_{K_\varepsilon}(t)-\beta_{n}(t)) = q(\gamma(t)- \gamma|_{K_\varepsilon}(s_j)) < \tau,\quad \forall n\geq n_{\tau,q}\] 
    and $\gamma(1)-\beta_{n}(1)=0$. Therefore
    \[ \sup_{t\in K_\varepsilon} q(\gamma|_{K_\varepsilon}(t)-\beta_{n}(t))\leq\tau,\quad \forall n\geq n_{\tau,q}.\] 
    For the second part, let us assume that $\gamma:[0,1]\to E$ verifies that for each $\varepsilon>0$, there exists $A_\varepsilon\in \mathcal{L}([0,1])$ such that $\lambda([0,1]\setminus A_\varepsilon)<\varepsilon/3$ and $\gamma|_{A_\varepsilon}$ is the uniform limit of the sequence $(\beta_n)_n$ of simple functions. By Lemma \ref{simple}, each simple function is Lusin-measurable and by Lemma \ref{Ksequence}, there exists a compact subset $H_\varepsilon$ of $[0,1]$ such that $\lambda([0,1]\setminus H_\varepsilon)<\varepsilon/3$ and each $\beta_n|_{H_\varepsilon}$ is continuous. By regularity of the Lebesgue measure, there exists a compact subset $K_\varepsilon$ of $[0,1]$ such that $K_\varepsilon\subseteq A_\varepsilon$ and $\lambda(A_\varepsilon\setminus K_\varepsilon)<\varepsilon/3$. Therefore, if $F_\varepsilon=H_\varepsilon\cap K_\varepsilon$, we have
    \[ \lambda([0,1]\setminus K_\varepsilon) \leq \lambda([0,1]\setminus A_\varepsilon) +\lambda(A_\varepsilon\setminus K_\varepsilon)<2\varepsilon/3\]
    and 
    \[ \lambda([0,1]\setminus F_\varepsilon) \leq \lambda([0,1]\setminus H_\varepsilon) +\lambda([0,1]\setminus K_\varepsilon)<\varepsilon.\]
    Moreover, since $(\beta|_{F_\varepsilon})_n$ is a sequence of continuous functions which converge uniformly to $\gamma|_{F_\varepsilon}$, the function $\gamma|_{F_\varepsilon}$ is continuous and hence, $\gamma$ is Lusin-measurable.
\end{proof}

\begin{coro}\label{1.8}
    Assume that $E$ is a locally convex space and $\gamma:[0,1]\to E$ a Lusin-measurable function. Then, for each $\varepsilon>0$ and each continuous seminorm $q$ on $E$, there exists a compact subset $K_{\varepsilon}$ of $[0,1]$ and a simple function $\beta_{\varepsilon,q}:[0,1]\to E$ such that $\lambda([0,1]\setminus K_{\varepsilon})<\varepsilon$, the function $\gamma|_{K_\varepsilon}$ is continuous and 
    \[ \sup_{t\in K_\varepsilon} q(\gamma(t)-\beta_{\varepsilon,q}(t))< \varepsilon. \]
\end{coro}
\begin{proof}
    Following the proof of Theorem \ref{equi}, for each $\varepsilon>0$ and each continuous seminorm $q$ on $E$, there exists a compact subset $K_\varepsilon$ of $[0,1]$ and a sequence $(\beta_n)_n$ of simple functions which converge uniformly to $\gamma$ in $K_\varepsilon$, i.e., there exists $n_{\varepsilon,q}\in \N$ such that
    \[ \sup_{t\in K_\varepsilon} q(\gamma(t)-\beta_{n}(t))< \varepsilon,\quad \forall n\geq n_{\varepsilon,q}. \]
    Hence $\beta_{\varepsilon,q}=\beta_{n_{\varepsilon,q}}$.
\end{proof}

\begin{coro}
Let $E$ be a locally convex space. If a function $\gamma:[0,1]\to E$ is uniform limit a.e. of simple functions, then it is Lusin-measurable.    
\end{coro}
\begin{proof}
    Let $(\beta_n)_n$ be a sequence of simple functions and $B\in \mathcal{L}([0,1])$ with $\lambda(B)=0$ such that $(\beta_n)_n$ converges uniformly to $\gamma$ on $[0,1]\setminus B$. Then for each $\varepsilon>0$, is enough to consider $A_\varepsilon=[0,1]\setminus B$ and the conclusion follows directly from Theorem \ref{equi}.
\end{proof}

\begin{coro}\label{stronglusin}
    Let $E$ be a locally convex space. If the function $\gamma:[0,1]\to E$ is Lusin-measurable, then is strongly-measurable.
\end{coro}
\begin{proof}
    By Theorem \ref{equi}, for each $j\in \N$, there exists a set $A_j\in \mathcal{L}([0,1])$ such that $\lambda([0,1]\setminus A_j)<1/j$ and there exists a sequence of simple functions $(\beta_{n,j})_n$ that converges uniformly to $\gamma$ on $A_j$. We disjointify the sequence $(A_j)_j$: Let $D_1=A_1$ and 
    \[ D_j = A_j \setminus\left( \bigcup_{i=1}^{j-1}A_i\right),\quad \forall j\in \N.\]    
    Then $(D_j)_j$ is a sequence of disjoint sets with $D_j\subseteq A_j$. We set
    \[ A=\bigcup_{j=1}^\infty D_j = \bigcup_{j=1}^\infty A_j.\]
    Hence $A\in \mathcal{L}([0,1])$ and 
    \[ \lambda([0,1]\setminus A) \leq\lambda([0,1]\setminus A_j) < 1/j,\quad \forall j\in \N.\] 
    Therefore, $\lambda([0,1]\setminus A) =0$. For each $n\in \N$, we define the simple function
    \[ \alpha_n:[0,1]\to E,\quad \alpha_n(t)=\sum_{i=1}^n \beta_{n,i}(t)\chi_{D_i}(t).\]
    We notice that if $t\in [0,1]\setminus A$, then $\alpha_n(t)=0$. Let $t\in A$ be fixed. Then there exists a unique $k\in \N$ such that $t\in D_k\subseteq A_k$. Therefore $\alpha_n(t)=0$ if $n<k$ and
    \[ \alpha_n(t)=\beta_{n,k}(t),\quad \forall n\geq k.\]
    By hypothesis, $\beta_{n,k}(t)\to \gamma(t)$ as $n\to \infty$. Hence, the sequence of simple functions $(\alpha_n)_n$ converges pointwise a.e. to $\gamma$.
\end{proof}

We recall \cite[Theorem 7.5.1.]{Egorov}.
\begin{theorem}(Egorov's Theorem) Let $(S,d)$ be a separable metric space,  $B\in \mathcal{L}([0,1])$ and $\gamma:B\to S$ be a Borel-measurable function. If $(\gamma_n)_n$ is a sequence of Lebesgue-measurable functions that converges pointwise a.e. to $\gamma$, then for each $\varepsilon>0$, there exists a subset $A_\varepsilon\in \mathcal{L}(B)$ with $\lambda(B\setminus A_\varepsilon)<\varepsilon$ such that $(\gamma_n)_n$ converges uniformly to $\gamma$ on $A_\varepsilon$.
\end{theorem}

\begin{coro}
    Let $E$ be a metrizable locally convex space. Then a function $\gamma:[0,1]\to E$ is strongly-measurable function if, and only if, is Lusin-measurable.
\end{coro}
\begin{proof}
    Let us consider the metric $d$ generated by a countable family of continuous seminorms $\mathcal{P}$ on $E$ (see e.g. \cite{GaN}). Let $(\beta_n)$ be a sequence of simple functions (and hence, Borel-measurable) that converges pointwise to $\gamma$ on $B\in\mathcal{L}([0,1])$ with $\lambda([0,1]\setminus B)=0$. Since for each $n\in \N$, the set $\beta_n([0,1])$ is finite, the set
    \[ S=\overline{\cup_{n=1}^\infty \beta_n([0,1])}.\] 
    endowed with $d$ is a separable metric space and $\gamma(B)\subseteq S$. Therefore, we consider the function $\gamma|_B:B\to S$. Since $\gamma|_B$ is strongly-measurable and $S$ metrizable, the function $\gamma|_B$ is Lebesgue-measurable. Hence, by Egorov's theorem, we have that for each $\varepsilon>0$, there exists $A_\varepsilon\in \mathcal{L}(B)$ with $\lambda(B\setminus A_\varepsilon)<\varepsilon$ such that $(\beta_n)_n$ converges uniformly to $\gamma|_B$ on $A_\varepsilon$. Since $A_\varepsilon\subseteq B\subseteq [0,1]$, we see that
    \[ \lambda([0,1]\setminus A_\varepsilon) \leq \lambda([0,1]\setminus B)+\lambda(B\setminus A_\varepsilon) < \varepsilon.\]
    Hence, by Theorem \ref{equi}, the function $\gamma$ is Lusin-measurable. The second part of the proof is given by Corollary \ref{stronglusin}.
\end{proof}

\begin{remark}\label{example}
    In notmetrizable locally convex spaces, strongly-measurability does not imply Lusin-measurability. Indeed, let us consider the real vector space 
    \[ E=\{f:\mathbb{R}\to \mathbb{R}\},\] 
    and for each element $x_0\in \mathbb{R}$, we define the seminorm
    \[ q_{x_0} (f) = \lvert f(x_0)\lvert ,\quad \forall f\in E.\]
    Then we endow the space $E$ with the locally convex Hausdorff topology generated by the family of separating seminorms $S=\{q_{x_0} : x_0\in \mathbb{R}\}$. Since $\R$ is not countable, we notice that $E$ is notmetrizable. We consider the function $\gamma:[0,1]\to E$ given by
    \begin{align*}
        \gamma(t)(x) &= \left\{\begin{array}{ll}
        1 & \text{, if } x=t \\
        0 & \text{, if } x\neq t,
    \end{array}\right.
    \end{align*}
    for each $x\in \mathbb{R}.$ Then $\gamma$ is not Lusin-measurable. Indeed, let $K\subseteq [0,1]$ be a compact subset such that $\gamma|_K$ is continuous. Thus, by the the definition of the product topology, the map
    \[ \phi_x:K\to \mathbb{R},\quad s\mapsto \gamma(s)(x)=\delta_s (x)\] 
    is continuous. Moreover, if $x\not\in K$ then $\phi_x=0$ and if $x\in K$ the set
    \[ \phi_x^{-1}\left(\left(\frac{1}{2},\infty\right)\right) = \{ x\}\]
    is open in $K$. Now, since $K$ is equipped with the discrete topology, compactness implies that $K$ is finite. Therefore, $\lambda(K)=0$ and 
    \[ \lambda([0,1]\setminus K ) = 1.\]
    Thus, $\gamma$ cannot be Lusin-measurable. To show that $\gamma$ is strongly-measurable, we need a sequence of simple functions 
    \[ \beta_n:[0,1]\to E,\quad \beta_n(t)=\sum_{i=1}^{m_n} a_{i,n} \chi_{A_{i,n}}(t),\]
    where $a_{1,n},...,a_{m_n,n}\in E$ and $A_{1,n},...,A_{m_n,n}\in \mathcal{L}([0,1])$, that converge pointwise a.e. to $\gamma$. Let us consider
    \[ a_{i,n}=\chi_{A_{i,n}}\in E,\]
    for all $n\in \N$ and $i\in\{1,...,m_n\}$. Hence, for each $t\in [0,1]$ and $x\in \R$, we have
    \[ \beta_n(t)(x)=\sum_{i=1}^{m_n} \chi_{A_{i,n}}(x) \chi_{A_{i,n}}(t).\]
    For $n\in \N$ and $i\in\{1,...,2^n-1\}$ we set 
    \[ A_{i,n}=\left[\frac{i-1}{2^n},\frac{i}{2^n}\right)\]
    and $A_{2^n,n}=\left[\frac{2^n-1}{2^n},1\right]$. Therefore, if $t=x$, we have
    \[ \beta_n(t)(x)=\sum_{i=1}^{2^n} \chi_{A_{i,n}}(x) \chi_{A_{i,n}}(t) = 1.\]
    If $t\neq x$, then $\delta=\lvert t-x\lvert >0$. Let $m_\delta\in \N$ such that $\frac{1}{2^{m_\delta}}<\delta$, then there exists an unique $j\in\{1,...,2^{m_\delta}\}$ such that $t\in A_{j,m_\delta}$ (and in this case $x\not\in A_{j,m_\delta}$). Therefore, we have 
    \[ \beta_n(t)(x)=\sum_{i=1}^{2^n} \chi_{A_{i,n}}(x) \chi_{A_{i,n}}(t) = 0,\quad \forall n>m_\delta\]
    In consequence, we have that $(\beta_n)_n$ converges pointwise to $\gamma$. Thus, $\gamma$ is strongly-measurable.
\end{remark}

\begin{remark}\label{pointwise}
    Let $E$ be a locally convex space. For $n\in \mathbb{N}$, let $\gamma_n:[0,1]\to E$ be a sequence of Lusin-measurable functions. We denote the pointwise limit of this sequence by
    \[ \gamma:[0,1]\to E,\quad \gamma(t) = \lim_{n\to \infty} \gamma_n(t). \] 
    Then, we observe that even though each $\gamma_n$ is continuous, the pointwise limit may not be Lusin-measurable. Indeed, for each $n\in \mathbb{N}$ we define
    \[ \gamma_{n}:[0,1]\to E,\quad t\mapsto \gamma_{n}(t).\] 
    Where
    \begin{align*}
    \gamma_{n}(t)(x) &=\left\{ \begin{array}{ll}
       1-n\lvert x-t \lvert  & \text{, if } x\in \left[t-1/n, t+1/n\right]  \\
       0  & \text{, if } x\not\in \left[t-1/n, t+1/n\right] 
    \end{array}\right.\\
    &= \max \{ 1-n\lvert x-t\lvert, 0\},
    \end{align*} 
    for each $x\in \mathbb{R}.$
    It is easy to see that $(\gamma_n)_n$ is a sequence in $C([0,1],E)$. Indeed, for each $n\in \mathbb{N}$ and $x\in \mathbb{R}$ the map
    \[ \phi_{n,x}:[0,1]\to \mathbb{R},\quad \phi_{n,x}(t)=\max \{ 1-n\lvert x-t\lvert, 0\}\] 
    is continuous. We identify $E$ with $\mathbb{R}^{\mathbb{R}}$  with the product topology, and for each $x_0\in \mathbb{R}$  we denote the projection 
    \[ \pi_{x_0}: \prod_{x\in \mathbb{R}} \mathbb{R}\to \mathbb{R},\quad (y_x)_{x\in\mathbb{R}}\mapsto y_{x_0}.\]
    Then for all $x_0\in\mathbb{R}$ and $t\in[0,1]$ we have that the function
    \[ \pi_{x_0}\circ \gamma_n:[0,1]\to \mathbb{R},\quad t\mapsto \pi_{x_0}\circ \gamma_n(t) = \phi_{n,x_0}(t) \] 
    is continuous. Hence  each $\gamma_n:[0,1]\to E$ is continuous and consequently, Lusin-measurable. Moreover, this sequence converges pointwise to the function
    \[ \gamma:[0,1]\to E,\quad t\mapsto \gamma(t).\]
    Where 
    \begin{align*}
        \gamma(t)(x) &= \left\{\begin{array}{ll}
        1 & \text{, if } x=t \\
        0 & \text{, if } x\neq t,
    \end{array}\right.
    \end{align*}
    for each $x\in \mathbb{R}$, which is not Lusin-measurable.
\end{remark}

We note that, under the assumption that the sequence $(\gamma_n)_n$ is a uniformly Cauchy sequence on an appropriate compact set, we obtain the following.

\begin{definition}
    Let $E$ be a locally convex space. We say a sequence of Lusin-measurable functions $(\gamma_n)_n$ is a uniformly Cauchy sequence in a subset $M$ of $[0,1]$ if for each $\varepsilon>0$ and each continuous seminorm $q$ on $E$, there exists $n_{\varepsilon,q}\in \N$ such that 
    \[ \sup_{t\in M} q\left( \gamma_m(t)-\gamma_n(t)\right) <\varepsilon,\quad \forall n,m\geq n_{\varepsilon,q}.\]
\end{definition}

\begin{lemma}\label{lemmapointwise} 
    Let $E$ be a sequentially complete Hausdorff locally convex space and $(\gamma_n)_n$ be a sequence of Lusin-measurable functions.  Assume that for each $\varepsilon>0$, there exists a compact subset $M_\varepsilon$ of $[0,1]$ such that $\lambda([0,1]\setminus M_\varepsilon)<\varepsilon$ and $(\gamma_n|_{M_\varepsilon})_n$ is a uniformly Cauchy sequence in $M_\varepsilon.$ Then the pointwise limit of $(\gamma_n)_n$ exists a.e. Moreover, the function 
    \[ \gamma:[0,1]\to E,\quad t\mapsto \gamma(t)=\left\{\begin{array}{rl}
       \lim_{n\to \infty} \gamma_n(t),  & \text{if the limit exists.} \\
        0, &  \text{if not.}
    \end{array}\right.\]
    is Lusin-measurable.
\end{lemma}

\begin{proof}
    Let $k \in \mathbb{N}$. By Lemma \ref{Ksequence}, there exists a compact subset $H_k$ of $[0,1]$ such that $\lambda([0,1]\setminus H_k )<1/2k$ and each $\gamma_n|_{H_k}$ is continuous. By hypothesis, there exists a compact subset $M_k$ of $[0,1]$ such that $\lambda([0,1]\setminus M_k )<1/2k$ and $(\gamma_n)$ is uniformly Cauchy in $M_k$, i.e., for each continuous seminorm $q$ on $E$, there exists $n_{k,q}\in \mathbb{N}$ such that
    \[ \sup_{t\in M_k} q\Big(\gamma_{m}(t)-\gamma_n(t)\Big)< \frac{1}{2k},\quad \forall m,n\geq n_{k,q}. \]
    We define the compact set $F_k=H_k\cap M_k$, then $\lambda([0,1]\setminus F_k)<1/k$ and $(\gamma_n|_{F_k})_n$ is a uniformly Cauchy sequence of continuous functions. Let us consider the Hausdorff locally convex space of continuous functions $C(F_k,E)$ with the uniform convergence topology with respect to the continuous seminorms of $E$. Since $E$ is sequentially complete, the space $C(F_k,E)$ is also sequentially complete. Then, the sequence $(\gamma_n|_{F_k})_n$ converges uniformly to a function $\alpha_k \in C(F_k,E)$. Since $E$ is a Hausdorff topological space, we observe that if $t\in F_{k_1}\cap F_{k_2}$, then $\alpha_{k_1}(t)=\alpha_{k_2}(t)$ by uniqueness of the limit of $(\gamma_n(t))_n$. We define the measurable set
    \[ F=\bigcup_{k=1}^\infty F_k.\] 
    Then 
    \[ \lambda([0,1]\setminus F) \leq \lambda([0,1]\setminus F_k) \leq \frac{1}{k},\quad \forall k\in \mathbb{N}. \] 
    Therefore $\lambda([0,1]\setminus F)=0$. We define the function
    \[ \gamma:[0,1]\to E,\quad t\mapsto \gamma(t)=\left\{ \begin{array}{rl}
       \alpha_k(t),  & \text{for any $k\in \N$ such that $t\in F_k$.}  \\
       0,  & \text{if $t\not\in F$.} 
    \end{array}\right.\]
    Then the function $\gamma:[0,1]\to E$ is well defined, it is the pointwise limit a.e. of the sequence $(\gamma_n)_n$ and is Lusin-measurable. Indeed, for each $\varepsilon>0$, there exists a $n_\varepsilon\in \mathbb{N}$ with $1/n_\varepsilon\leq \varepsilon$ such that $\lambda([0,1]\setminus F_{n_\varepsilon})<\varepsilon$ and $\gamma|_{F_{n_\varepsilon}}=\alpha_{n_\varepsilon}\in C(F_{n_\varepsilon},E)$.
\end{proof}

\begin{example}
    Let $E$ be a sequentially complete Hausdorff locally convex space and $(\gamma_n)_n$ be a Cauchy sequence in $C([0,1],E)$. If $\varepsilon>0$, then applying Lemma \ref{lemmapointwise} and assuming  $H_\varepsilon = [0,1]$, we have that the pointwise limit $\gamma$ is Lusin-measurable.
\end{example}

\begin{remark}
    We conclude this section highlighting the hierarchy of measurability for a locally convex space
    \[ \left\{ \text{Lusin-measurable functions} \right\} \subset \left\{ \text{strongly-measurable functions} \right\} \subset \left\{ \text{weakly-measurable functions} \right\}.\]
\end{remark}
\section{$L^p$-integrable functions} 

\begin{definition}
Let $E$ be a locally convex space and $1\leq p<\infty$. We define the space $\mathcal{L}^p([0,1],E)$ to be the set of all Lusin-measurable functions $\gamma : [0,1]\to E$ such that for each continuous seminorm $q$ on $E$ we have that
\[q\circ \gamma \in \mathcal{L}^p([0,1],\mathbb{R}).\] 
If $\gamma,\eta:[0,1]\to E$ are two Lusin-measurable functions, we say $\gamma \sim \eta$ if and only if $\gamma(t) = \eta(t)$ for almost all $t\in [0,1]$ and write $[\gamma]$ for the equivalence class of $\gamma$. This relation enable us to define the space
\[ L^p([0,1],E)= \mathcal{L}^p([0,1],E) / [0].\]
For each continuous seminorm $q$ of $E$, we define the seminorm
\[ \lVert [\gamma]\lVert_{\mathcal{L}^p,q}=\left( \int_0^1 q(\gamma(t))^pdt \right)^{1/p},\quad \forall [\gamma]\in L^p([0,1],E).\]
Then these seminorms define a locally convex topology for $L^p([0,1],E)$. If no confusion arises, we will denote $[\gamma]$ simply as $\gamma$. For the case $p=\infty$, using the essential supremum, it is possible to define the space $L^\infty([0,1],E)$, but we will only focus on the finite case.
\end{definition}

\begin{remark}\label{remark}
    Let $E$ be a Hausdorff locally convex space. We recall that a Lusin-measurable function $\gamma:[0,1]\to E$ verifies $\gamma(t)=0$ a.e. if, and only if, $q(\gamma(t))=0$ a.e., for each continuous seminorm $q$ on $E$ (see e.g. \cite{FPM, Nin1}). Therefore, if $E$ is a Hausdorff topological space, the space $L^p([0,1],E)$ is also Hausdorff.
\end{remark}

\begin{remark}
 In Banach settings, using the notion of strongly-measurability, the Bochner integral has been extensively studied. In locally convex setting, the previous definition cannot define a Hausdorff topological space using strongly-measurability. Indeed, let us consider the Hausdorff locally convex space $E=\{f:\R\to \R\}$ and the strongly-measurable function $\gamma:[0,1]\to E$ as in Remark \ref{example}.  Then, for each $x\in \R$, we have
\[ \lVert \gamma \lVert_{\mathcal{L}^p,q_x}^p=\int_0^1 q_x(\gamma(t))dt = \int_0^1 \delta_t(x) dt  = 0.\]
Moreover, we notice that
\[ \lambda\Big(\{t\in [0,1] : \gamma(t)\neq0_E\}\Big)=1. \] 
Therefore, the family of seminorms $\{\lVert\cdot\lVert_{\mathcal{L}^p,q_x}: x\in \R\}$ cannot separate points on the $L^p$-space constructed with strongly-measurable functions.
\end{remark}

\begin{remark}
    We recall the topological embeddings 
    \[ C([0,1],E)\hookrightarrow L^\infty([0,1],E) \hookrightarrow L^r([0,1],E)\hookrightarrow L^p([0,1],E)\hookrightarrow L^1([0,1],E),\] 
    where $1\leq p\leq r <\infty$ (see e.g. \cite{Nin1}). Moreover, in \cite{FPM}, the authors show that if $E$ is sequentially complete, the space $L^\infty([0,1],E)$ is sequentially complete. Furthermore, if $E$ is quasi-complete and fundamentally $L^p$-bounded, then $L^p([0,1],E)$ is locally complete for $1\leq p<\infty$. 
\end{remark}

\begin{theorem}
    Let $1\leq p<\infty$ and let $E$ be a locally convex space. If $\gamma\in L^p([0,1],E)$, then for each $\varepsilon>0$ and each continuous seminorm $q$, there exists a simple function $\beta_{\varepsilon,q}:[0,1]\to E$ such that 
    \[ \lVert \gamma-\beta_{\varepsilon,q}\lVert_{\mathcal{L}^p,q} < \varepsilon.\] 
 Consequently, the space of simple functions $\mathcal{S}([0,1],E)$ is dense in $L^p([0,1],E)$.
\end{theorem}
\begin{proof}
    Let $q$ be a continuous seminorm of $E$ and $\varepsilon>0$. By absolute continuity of the integral, there exists $\delta=\delta_{\gamma,\varepsilon,q}>0$ such that for every measurable set $A\in \mathcal{L}([0,1])$
    \[ \lambda(A)<\delta \implies \int_A q(\gamma(t))^pdt< \frac{\varepsilon^p}{2}.\] 
    Moreover, by Corollary \ref{1.8}, for $\theta_\varepsilon=\min\left\{\delta,\varepsilon/\sqrt[p]{2}\right\}$ there exists a compact subset $K_{\varepsilon}$ of $[0,1]$ and a simple function $\beta_{\varepsilon,q}:[0,1]\to E$ such that $\lambda([0,1]\setminus K_\varepsilon)<\theta_\varepsilon$, the function $\gamma|_{K_\varepsilon}$ is continuous and
    \[ \sup_{t\in K_\varepsilon} q(\gamma(t)-\beta_{\varepsilon,q}(t))<\theta_\varepsilon \leq\frac{\varepsilon}{\sqrt[p]{2}}\] 
    Without loss of generality, we suppose that $\beta_{\varepsilon,q}(t)=0$ for all $t\in [0,1]\setminus K_\varepsilon$. Then 
    \begin{align*}
        \lVert \gamma-\beta_{\varepsilon,q}\lVert_{\mathcal{L}^p,q}^p &= \int_0^1 q(\gamma(t)-\beta_{\varepsilon,q}(t))^p dt \\
        & = \int_{K_\varepsilon}q(\gamma(t)-\beta_{\varepsilon,q}(t))^p dt + \int_{[0,1]\setminus K_\varepsilon} q(\gamma(t)-\beta_{\varepsilon,q}(t))^pdt\\
        &\leq \sup_{t\in K_\varepsilon} q(\gamma(t)-\beta_{\varepsilon,q}(t))^p\lambda(K_{\varepsilon}) + \int_{[0,1]\setminus K_\varepsilon} q(\gamma(t))^p dt\\
        & \leq \frac{\varepsilon^p}{2}+\frac{\varepsilon^p}{2} = \varepsilon^p.
    \end{align*}
    Hence $ \lVert \gamma-\beta_{\varepsilon,q}\lVert_{\mathcal{L}^p,q} < \varepsilon$. Let $\gamma\in L^p([0,1],E).$ Then a basic open $\gamma$-neighborhood is a set of the form
    \[ V_\gamma = \gamma+\bigcap_{i=1}^n \lVert \cdot \lVert_{\mathcal{L}^p,q_i}^{-1}\left([0,\varepsilon)\right)\] 
 for the continuous seminorms $\varepsilon>0$ and $q_1,...,q_n$ of $E$. \\
 Next, we define the continuous seminorm $q:E\to [0,\infty)$ by
    \[ q(x)=\max\{q_1(x),...,q_n(x)\},\quad \forall x\in E\] 
    Then, there exists $\beta_{\varepsilon,q}\in \mathcal{S}([0,1],E)$ such that 
    \[ \lVert \gamma-\beta_{\varepsilon,q}\lVert_{\mathcal{L}^p,q}<\varepsilon.\]
    Therefore, for each $i\in \{1,...,n\}$, we have
    \[ \lVert \gamma-\beta_{\varepsilon,q}\lVert_{\mathcal{L}^p,q_i} \leq \lVert \gamma-\beta_{\varepsilon,q} \lVert_{\mathcal{L}^p,q}< \varepsilon.\] 
    Hence $\beta_{\varepsilon,q}\in V_\gamma$ and $\mathcal{S}([0,1],E)$ is dense.
\end{proof}

We recall \cite[Lemma 2.12]{Rud}.
\begin{lemma}[Urysohn's Lemma]
Let $X$ be a locally compact Hausdorff space. Suppose that $V\subseteq X$ is open and $K\subseteq V$ compact. Then there exists a continuous function with compact support $f:X\to \mathbb{R}$ such that 
\[ \chi_K \leq f \leq \chi_V.\]
\end{lemma}
\begin{theorem}\label{2.11}
   Let $1\leq p<\infty$ and $E$ be a Hausdorff locally convex space. Then the space of continuous functions $C([0,1],E)$ is dense in $L^p([0,1],E)$.
\end{theorem}
\begin{proof}
    Let $y_0\in E$ and $A\in\mathcal{L}([0,1])$. Since $\mathcal{S}([0,1],E)$ is dense in $L^p([0,1],E)$. Then it suffices to show that the simple function $y_0\chi_A:[0,1]\to E$ belongs to the closure of $C([0,1],E)$.\\
    In fact,
    \begin{align*}
        \lambda(A)&=\sup \{\lambda(K) : \text{$K$ is compact and $K\subseteq A$}\}\\
        &= \inf \{ \lambda(U) : \text{$U$ is open and $A\subseteq U$}\}.
    \end{align*}
    Then there exists a sequence of compact sets $(K_n)_n$ and open sets $(U_n)_n$ such that for each $n\in \mathbb{N}$, we have that $K_n\subseteq A\subseteq U_n$ and
    \[ \lambda(A\setminus K_n) <\frac{1}{2n}\quad \land \quad \lambda(U_n\setminus A)<\frac{1}{2n}.\]
    By Urysohn's Lemma, there exists a continuous function $f_n:[0,1]\to \mathbb{R}$ such that 
    \[  \chi_{K_n}\leq f_n\leq \chi_{U_n}. \] 
    Then the function 
    \[ \alpha_n:[0,1]\to E,\quad \alpha_n(t)=f_n(t) y_0, \] 
    is in $C([0,1],E)$. \\
    We see that if $t\in K_n$ or $t\in [0,1]\setminus U_n$ then $(\alpha_n - y_0\chi_{A})(t)=0$. Now, since 
    \[U_n\setminus K_n = (U_n\setminus A)\cup (A\setminus K_n).\]
    Therefore, for each continuous seminorm $q$ of $E$ we have that
    \begin{align*}
        \lVert \alpha_n-y_0\chi_A \lVert_{\mathcal{L}^p,q}^p &= \int_0^1 q( \alpha_n(t)-y_0\chi_A(t))^pdt  \\
        &= \int_{U_n\setminus K_n} q(f_n(t) y_0 - y_0 \chi_A(t))^p dt \\
        & = \int_{U_n\setminus A} q(f_n(t) y_0 - y_0 \chi_A(t))^p dt + \int_{A\setminus K_n} q(f_n(t) y_0 - y_0 \chi_A(t))^p dt \\
        &=  q( y_0 )^p\left( \int_{U_n\setminus A} \lvert f_n(t)\lvert^p  dt +  \int_{A\setminus K_n} \lvert f_n(t) -1\lvert^p dt\right) \\
        &\leq q(y_0)^p\left( \lambda(U_n\setminus A) + \lambda(A\setminus K_n) \right) \\
        &\leq q(y_0)^p \left(\frac{1}{2n}+ \frac{1}{2n}\right) = \frac{q(y_0)^p}{n}.
    \end{align*}
    Hence $\alpha_n \to y_0\chi_A$ as $n\to \infty$ in $L^p([0,1],E)$.
\end{proof}

\begin{remark}
    Let us consider the locally convex space $E=\{f:\mathbb{R}\to \mathbb{R}\}$ and the sequence $(\gamma_n)_n$ given in Remark \ref{pointwise}. Then  $(\gamma_n)_n$ is in $L^p([0,1],E)$. We claim that this sequence converges in the  $L^p([0,1],E)$ norm. In fact, let be  $0_e:\mathbb{R}\to \mathbb{R}$ and $0_E:[0,1]\to E$ be the respective constants maps $0.$  It is clearly that  $0_E\in L^p([0,1],E).$ Moreover,  for $x_0\in \mathbb{R}$ fixed we have that,
    \begin{align*}
        \lVert \gamma_n  - 0_E\lVert_{\mathcal{L}^p,q_{x_0}}^p &= \int_0^1 (q_{x_0} ( \gamma_n (t) ))^p dt \\
        &= \int_0^1 (\gamma_n (t)(x_0))^p dt\\
        &= \int_0^1 \max \{ 1-n\lvert x_0-t\lvert, 0\}^p dt\\
        &= \int_{\left[x_0-\frac{1}{n},x_0+\frac{1}{n}\right]\cap [0,1]} \left( 1-n\lvert x_0 -t \lvert\right)^p dt \\
        &\leq \int_{\left[x_0-\frac{1}{n},x_0+\frac{1}{n}\right]\cap [0,1]} 1 dt  \\
        &\leq \frac{2}{n}.
    \end{align*}
    Hence $\gamma_n \to 0_E$ in $L^p([0,1],E)$.
    On the other hand, let $\gamma:[0,1]\to E$ be the pointwise limit of the sequence $(\gamma_n)_n$ (see Remark \ref{pointwise}). Since
    \[ \lambda\Big(\{t\in [0,1] : \gamma(t)\neq0_E(t)\}\Big)=1, \] 
    we see that the pointwise limit is not equal a.e. to the limit in $L^p([0,1],E)$.
   Hence for functions with values in a non-metrizable locally convex space, the convergence in $L^p$ does not imply the existence of a subsequence which converges to the pointwise limit.  This shows that functions taking values in a non-metrizable space may exhibit behaviors that do not occur in the metrizable setting.
    
\end{remark}

\begin{definition}
    Let $E$ be a topological vector space and $\gamma:[0,1]\to E$ be a function. Then
    \begin{itemize}
        \item[a)] We say $\gamma$ is weak integrable if it is weak-measurable and for each $f\in E'$, we have 
        \[ f\circ \gamma\in \mathcal{L}^1([0,1]).\]
        \item[b)] We say $\gamma$ is Pettis integrable if it is weak integrable and there exists $z\in E$ such that 
    \[ f(z) = \int_0^1 f(\gamma(t))dt,\quad \forall f\in E'.\] 
    we denote this element by $z=\int_0^1 \gamma(t)dt$.
    \end{itemize}
\end{definition}
For properties of the Pettis integral on locally convex spaces we refer the reader to \cite{Ali, Chat, Pettis, Tho}.

\begin{remark}
    By the Hahn-Banach theorem, if $E$ is a Hausdorff locally convex space, then its topological dual $E'$ separates points. Hence, if $\gamma:[0,1]\to E$ is Pettis integrable, the element $\int_0^1\gamma(t)dt$ is necessarily unique.
\end{remark}
The following result (\cite[Proposition 2.26]{Nin1}) shows a sufficient condition for the existence of the Pettis integral for $L^1$-integrable functions.

\begin{theorem}\label{2weak}
    Let $E$ be a sequentially complete Hausdorff locally convex space. Then each function $\gamma\in \mathcal{L}^1([0,1],E)$ is Pettis integrable. Furthermore, the function
    \[ \eta:[0,1]\to E,\quad \eta(t)=\int_0^t \gamma(s)ds\] 
    is continuous.
\end{theorem}

We recall \cite[Exercise 1.2.16]{TTa}.

\begin{lemma}
    For each $n\in \mathbb{N}$, we denote the dyadic partition $\pi_n$ of $[0,1],$ that is, 
    \[\pi_n= \left\{ \left[\frac{k-1}{2^n},\frac{k}{2^n}\right] : k\in\{1,....,2^n\}\right\}.\]
    Let $A\in \mathcal{L}([0,1])$. Then for each $\varepsilon>0$, there exists an $n_\varepsilon\in \mathbb{N}$ and a finite collection $I_1,...,I_k\in \pi_{n_\varepsilon}$, such that 
    \[ \lambda\left(A\triangle\bigcup_{j=1}^k I_j\right)<\varepsilon,\]
    where $\triangle$ denotes the symmetric difference.
\end{lemma}

\begin{remark}\label{absctint}
    Let $\gamma\in L^p([0,1],E)$. By the previous lemma and the absolute continuity of the integral, for each $\varepsilon>0$ and each $A\in \mathcal{L}([0,1])$, there exists $n_{\varepsilon,A}\in \mathbb{N}$ for which there exists a finite collection $I_1,...,I_k\in \pi_{n_{\varepsilon,A}}$ satisfying
\[ \int_{A\triangle\bigcup_{j=1}^k I_j} q(\gamma(t))^p dt < \varepsilon.\] 
\end{remark}

\begin{lemma}\label{HBTHEOREM}
    Let $E$ be a locally convex space and $\gamma:[0,1]\to E$ be a  Pettis integrable function. Then, for each continuous seminorm $q$ of $E$ we have that
    \[ q\left( \int_0^1 \gamma(t)dt \right) \leq \int_0^1 q(\gamma(t))dt. \] 
\end{lemma}
\begin{proof}
    The Hahn-Banach Extension Theorem (see Corollary B.5.12 \cite{GaN}), implies that 
    \[ q(x) = \sup \left\{ \lvert f(x)\lvert : \left(f\in E': \lvert f\lvert \leq q\right)\right\}. \]
    Moreover, if $f\in E'$ then 
    \[ \left\lvert f\left( \int_0^1 \gamma(t)dt \right) \right\lvert = \left\lvert \int_0^1 f(\gamma(t))dt \right\lvert \leq \int_0^1 \lvert f(\gamma(t)) \lvert dt. \]
    Assume that for all $x\in E$, we have $\lvert f(x)\lvert \leq q(x).$ Then 
    \[ \left\lvert f\left( \int_0^1 \gamma(t)dt \right) \right\lvert  \leq \int_0^1  q(\gamma(t))  dt. \]  
    Now, taking the supremum with respect all those  $f\in E'$ such that $\lvert f(x)\lvert \leq q(x)$ for all $x\in E$ we have
    \[ q\left( \int_0^1 \gamma(t)dt \right) \leq \int_0^1 q(\gamma(t))dt \] for all the seminorms $q.$
\end{proof}

 Using the notion of measurability given by uniform approximation by nets of simple functions on measurable sets of arbitrarily large measure (as in Theorem \ref{equi}), the authors in \cite[Lemma 4]{PBr}, show an approximation result in $L^p$-spaces by nets of simple functions indexed by all finite measurable partitions of the domain. In our next result, we obtain a sequential approximation result using dyadic partitions on the closed interval.

\begin{theorem}
    Let $E$ be a sequentially complete Hausdorff locally convex space, $1\leq p <\infty$ and $\gamma\in L^p([0,1],E)$. For $n\in \mathbb{N}$, let $\pi_n$ the dyadic partition of $[0,1]$ and we define the sequence of simple functions by
    \[ \gamma_n:[0,1]\to E,\quad \gamma_n(t) = \sum_{j=1}^{2^n} \left( 2^n \int_{I_j} \gamma(t)dt\right) \chi_{I_j}(t).\]
    Then the sequence $(\gamma_n)_n$ converges to $\gamma$ in $L^p([0,1],E)$.
\end{theorem}
\begin{proof}
    Under the hypothesis that $E$ is sequentially complete. Then by Theorem \ref{2weak}, the simple function $\gamma_n$ are well defined. Since $\int_{I_j}\gamma(t)dt$ exists. Moreover, they are Lusin-measurable and $L^p$-integrable. \\
 Now, let $\varepsilon>0$ and let  $q$ be a continuous seminorm of $E$. Then there exists a simple function $\beta_{\varepsilon,q}:[0,1]\to E$ such that $\lVert \gamma-\beta_{\varepsilon,q}\lVert_{\mathcal{L}^p,q} < \varepsilon/3$. This simple function $\beta_{\varepsilon,q}$ is  Pettis integrable and if
    \[ \beta_{\varepsilon,q}=\sum_{i=1}^m y_i \chi_{A_i}, \] 
    where $y_1,...,y_m\in E$ and $A_1,...,A_m\in \mathcal{L}([0,1])$ form a measurable partition of $[0,1]$, then
    \[ \int_0^1 \beta_{\varepsilon,q}(t)dt = \sum_{i=1}^m  y_i \lambda(A_i).\] 
    Thus we have that
    \begin{align*}
        \lVert \gamma-\gamma_n \lVert_{\mathcal{L}^p,q} &\leq \lVert \gamma-\beta_{\varepsilon,q} \lVert_{\mathcal{L}^p,q} + \lVert \beta_{\varepsilon,q} - \gamma_n \lVert_{\mathcal{L}^p,q}.
    \end{align*}
    The first term on right hand side of the above inequality is already bounded by $\varepsilon/3.$ Thus,  we only need  to    estimate the second term,
    \begin{align*}
        \lVert \beta_{\varepsilon,q} - \gamma_n \lVert_{\mathcal{L}^p,q}^p &= \int_0^1 q\left( \sum_{i=1}^m y_i \chi_{A_i}(s)- \sum_{j=1}^{2^n} \left( 2^n \int_{I_j} \gamma(t)dt\right) \chi_{I_j}(s) \right)^p ds \\
        &= \sum_{i=1}^m \int_{A_i} q\left( y_i\chi_{A_i}(s)- \sum_{j=1}^{2^n} \left( 2^n \int_{I_j} \gamma(t)dt\right) \chi_{I_j}(s) \right)^p ds.
    \end{align*}
    By Remark \ref{absctint}, let $n_{\varepsilon,A_1,...,A_m}\in \mathbb{N}$ be such that for each $i\in \{1,....,m\}$, there exists a finite collection of intervals $I_{i,1},...,I_{i,k_i}\in \pi_{n_{\varepsilon,A_1,...,A_m}}$ that verify
    \[ \sum_{i=1}^m \int_{A_i \triangle \cup_{j=1}^{k_i} I_{i,j}} q\left( y_i\chi_{A_i}(s)- \sum_{j=1}^{2^n} \left( 2^n \int_{I_j} \gamma(t)dt\right) \chi_{I_j}(s) \right)^p ds < \varepsilon/3\]
    and 
    \[\sum_{i=1}^m q(y_i)\lambda\Big(A_i \triangle \cup_{j=1}^{k_i} I_{i,j}\Big) \leq \varepsilon/3.\]
    Since 
    \[ A_i \subseteq \left( A_i\cap \bigcup_{j=1}^{k_i} I_{i,j}\right) \cup \left( A_i \triangle \bigcup_{j=1}^{k_i} I_{i,j}\right). \]
    Hence for $n\geq n_{\varepsilon,A_1,...,A_m}$ we obtain that,
    \begin{align*}
        \lVert \beta_{\varepsilon,q} - \gamma_n \lVert_{\mathcal{L}^p,q}^p &\leq  \sum_{i=1}^m \sum_{j=1}^{k_i} \int_{A_i\cap I_{i,j}} q\left( y_i\chi_{A_i}(s)- \left( 2^n \int_{I_{i,j}} \gamma(t)dt\right) \chi_{I_j}(s) \right)^p ds + \varepsilon/3 \\
        &= \sum_{i=1}^m \sum_{j=1}^{k_i} q\left( y_i-  2^n \int_{I_{i,j}} \gamma(t)dt \right)^p \lambda\left(A_i\cap I_{i,j}\right) + \varepsilon/3 \\
        &\leq \sum_{i=1}^m \sum_{j=1}^{k_i} q\left( y_i-  2^n \int_{I_{i,j}} \gamma(t)dt \right)^p \frac{1}{2^n} + \varepsilon/3 \\
        &= \sum_{i=1}^m \sum_{j=1}^{k_i} q\left( 2^n \int_{ I_{i,j}} y_i\chi_{I_{i,k}}(t) dt- 2^n \int_{I_{i,j}} \gamma(t)dt \right)^p \frac{1}{2^n} + \varepsilon/3 \\
         &= \sum_{i=1}^m \sum_{j=1}^{k_i} q\left( 2^n \int_{ I_{i,j}} y_i\Big(\chi_{A_i\cap I_{i,k}}(t) + \chi_{I_{i,k}\setminus A_i}(t)\Big)dt- 2^n \int_{I_{i,j}} \gamma(t)dt \right)^p \frac{1}{2^n} + \varepsilon/3 \\
        &= \sum_{i=1}^m \sum_{j=1}^{k_i} q\left( 2^n \left(\int_{ I_{i,j}} (y_i\chi_{A_i}(t)dt -\gamma(t))dt\right) + 2^n \int_{ I_{i,j}} y_i\chi_{I_{i,k}\setminus A_i}(t)dt\right)^p \frac{1}{2^n} + \varepsilon/3 \\ 
        &\leq \sum_{i=1}^m \sum_{j=1}^{k_i} q\left(  \int_{ I_{i,j}} (y_i\chi_{A_i}(t)dt -\gamma(t))dt\right)^p \frac{2^{pn}}{2^n} + q\left(\int_{ I_{i,j}} y_i\chi_{I_{i,k}\setminus A_i}(t)dt\right)^p \frac{2^{pn}}{2^n} + \varepsilon/3 \\
        &= \sum_{i=1}^m \sum_{j=1}^{k_i} q\left(  \int_{ I_{i,j}} (\beta_{\varepsilon,q}(t)dt -\gamma(t))dt\right)^p \frac{2^{pn}}{2^n} + q\left(\int_{ I_{i,j}} y_i\chi_{I_{i,k}\setminus A_i}(t)dt\right)^p \frac{2^{pn}}{2^n} + \varepsilon/3.
    \end{align*}
    Under the condition $\frac{1}{p}+\frac{1}{r}=1.$ Then by Lemma \ref{HBTHEOREM} we have
    \begin{align*}
        \sum_{i=1}^m \sum_{j=1}^{k_i} q\left( \int_{I_{i,j}} \left(\beta_{\varepsilon,q}(t) - \gamma(t) \right)dt \right)^p \frac{2^{pn}}{2^n} &\leq \sum_{i=1}^m \sum_{j=1}^{k_i} \left(  \int_{I_{i,j}} q\left( \beta_{\varepsilon,q}(t)-\gamma(t)\right) dt \right)^p \frac{2^{pn}}{2^n}\\
        &\leq \sum_{i=1}^m \sum_{j=1}^{k_i} \left( \int_{I_{i,j}} q(\beta_{\varepsilon,q}(s)-\gamma(s))^p dt \right)^{p/p}  \left(\int_{I_{i,j}} 1 dt \right)^{p/r} \frac{2^{pn}}{2^n}\\
        &=  \sum_{i=1}^m \sum_{j=1}^{k_i} \left( \int_{I_{i,j}} q(\beta_{\varepsilon,q}(s)-\gamma(s))^p dt \right)\frac{2^{pn}}{2^{np/r}2^n}\\
        &= \left(\int_0^1 q(\beta_{\varepsilon,q}(s)-\gamma(s))^p dt\right) 2^0\\
        & \leq \varepsilon/3.
    \end{align*}
    On the other hand
    \begin{align*}
        \sum_{i=1}^m \sum_{j=1}^{k_i}  q\left(\int_{ I_{i,j}} y_i\chi_{I_{i,k}\setminus A_i}(t)dt\right)^p \frac{2^{pn}}{2^n} &\leq \sum_{i=1}^m \sum_{j=1}^{k_i}  \left( \int_{ I_{i,j}} q\Big(y_i\chi_{I_{i,k}\setminus A_i}(t)\Big)dt\right)^p \frac{2^{pn}}{2^n} \\
        &\leq \sum_{i=1}^m \sum_{j=1}^{k_i} \left( \int_{ I_{i,j}} q\Big(y_i\chi_{I_{i,k}\setminus A_i}(t)\Big)dt \right)^{p/p}  \left(\int_{I_{i,j}} 1 dt \right)^{p/r} \frac{2^{pn}}{2^n}\\
        \quad \quad &\leq \sum_{i=1}^m  \int_{\cup_{j=1}^{k_i}I_{i,j}} q\Big(y_i\chi_{\cup_{j=1}^{k_i}I_{i,k}\setminus A_i}(t)\Big)dt \\
        & = \sum_{i=1}^m  \int_{\cup_{j=1}^{k_i}I_{i,k}\setminus A_i} q(y_i)\chi_{\cup_{j=1}^{k_i}I_{i,k}\setminus A_i}(t)dt \\
        &\leq \sum_{i=1}^m q(y_i)\lambda\Big(A_i \triangle \cup_{j=1}^{k_i} I_{i,j}\Big)\\
        &\leq \varepsilon/3.
    \end{align*}
    Therefore
    \[ \lVert \gamma -\gamma_n \lVert_{\mathcal{L}^p,q} \leq \varepsilon,\quad \forall n\geq n_\varepsilon. \]
\end{proof}

\noindent
Pinaud, Matthieu F; Departamento de Matem\'atica y Ciencia de la Computaci\'on, Universidad de Santiago de Chile, Casilla 307, Correo 2, Santiago, Chile; matthieu.pinaud@usach.cl\\

\noindent
Prado, Humberto; Departamento de Matem\'atica y Ciencia de la Computaci\'on, Universidad de Santiago de Chile, Casilla 307, Correo 2, Santiago, Chile;  humberto.prado@usach.cl
\end{document}